\documentclass[11pt,reqno]{amsart}
\usepackage{amscd,amssymb,amsmath,amsthm}
\usepackage[arrow,matrix]{xy}
\usepackage{graphicx}
\usepackage{cite}

\newtheorem{thm}[subsection]{Theorem}
\newtheorem{lemma}[subsection]{Lemma}
\newtheorem{pro}[subsection]{Proposition}
\newtheorem{cor}[subsection]{Corollary}

\numberwithin{equation}{section} 

\newcommand{\G}{{\mathcal G}}

\def\cb{{\mathcal B}}

\newcommand{\w}{{\bf w}}
\newcommand{\s}{{\sigma}}

\newcommand{\bea}{\begin{eqnarray}}
\newcommand{\eea}{\end{eqnarray}}

\newcommand{\Z}{\mathbb{Z}}
\newcommand{\Q}{\mathbb{Q}}

\def\ce{{\mathcal E}}
\def\cg{{\mathcal G}}

\def\bn{{\mathbb N}}

\def\bq{{\mathbb Q}}

  \def\G{\Gamma}

\def\k{\kappa}
\def\m{\mu}

\def\s{\sigma} 
\def\t{\theta}

\def\w{\omega} \def\Om{\Omega}

\def\h{{\mathbf{h}}}

\def\xb{{\mathbf{x}}}

\begin{document}

\title[$p$-adic Potts-Bethe mapping]
{Chaotic behavior of the $P$-adic Potts-Bethe mapping}

\author{Farrukh Mukhamedov}
\address{Farrukh Mukhamedov\\
 Department of Mathematical Sciences\\
College of Science, The United Arab Emirates University\\
P.O. Box, 15551, Al Ain\\
Abu Dhabi, UAE} \email{{\tt far75m@gmail.com} {\tt
farrukh.m@uaeu.ac.ae}}

\author{ Otabek Khakimov}
\address{ Otabek Khakimov\\
Institute of Mathematics, National University of Uzbekistan, 29,
Do'rmon Yo'li str., 100125, Tashkent, Uzbekistan.} \email {{\tt
hakimovo@mail.ru}}

\begin{abstract}
In our previous investigations, we have developed the
renormalization group method to $p$-adic models on Cayley trees,
this method is closely related to the investigation of $p$-adic
dynamical systems associated with a given model. In this paper, we
study chaotic behavior of the Potts-Bethe mapping. We point out that
a similar kind of result is not known in the case of real numbers
(with rigorous proofs).
 \vskip 0.3cm \noindent {\it
Mathematics Subject Classification}: 37B05, 37B10,12J12, 39A70\\
{\it Key words}: $p$-adic numbers, Ising-Potts function; chaos;
shift.
\end{abstract}

\maketitle

\section{introduction}

In \cite{Egg} it was studied the thermodynamic behavior of the
central site of an Ising spin system with ferromagnetic
nearest-neighbor interactions on a Cayley tree by recursive methods.
This method opened new perspectives between the recursion approach
and the theory of dynamical systems. The existence of the phase
transition is closely connected to the existence of chaotic behavior
of the associated dynamical system which is governed by the
Potts-Bethe function.  Investigation the dynamics of this function
has been the object of no small amount of study in the real and
complex settings. This deceptively simple family of rational
functions has given rise to a surprising number of interesting
dynamical features (see for example \cite{ADH,BG, FTC,Kap,Monr}).

On the  other hand, recently, polynomials and rational maps of
$p$-adic numbers $\bq_p$ have been studied as dynamical systems over
this field \cite{B0,B1}. It turns out that these $p$-adic dynamical
systems are quite different to the dynamical systems in Euclidean
spaces (see for example, \cite{AKh,FL2,HY,KhN,L,Sil1} and their
bibliographies therein). In theoretical physics, the interest in
$p$-adic dynamical systems was started with the development of
$p$-adic models \cite{M13,MK16,KhN}. In these investigations, it was
stresses the importance of the detecting chaos in the $p$-adic
setting \cite{Kh22,TVW,Wo}. In \cite{M15} it has been developed the
renormalization group method to study phase transitions for several
$p$-adic models on Cayley trees. In \cite{MR1,M12,M13,MA131} we
studied $q$-state $p$-adic Potts model on Cayley tree of order $k$
(or Bethe lattice). To investigate the phase transition for this
model, we need to investigated dynamical behavior of the
\textit{Potts-Bethe mapping} (see Appendix)
\begin{equation}\label{11}
f_\theta(x)=\left(\frac{\theta x+q-1}{x+\theta+q-2}\right)^k.
\end{equation}
Most of the results have been obtained in case $k=2$ and $q$ is
devisable by $p$, i.e. $|q|_p<1$ (see \cite{MR1}). Namely, in
\cite{MK16} we have prove that the existence of the phase transition
for the model implies the chaoticity of the Potts-Bethe mapping (in
the mentioned regime). For the case $k=3$ and $p\equiv 2 (mod 3)$,
the stability of the mapping \eqref{11} have been studied in
\cite{SA16}. Note that case $k=1$ it has been investigated in
\cite{FL3,MR0}.

In the present paper, we study dynamics of the $p$-adic Potts-Bethe
mapping in a general setting. Namely, we consider $k$ is arbitrary
and $q$ is not divisible $p$ (The case when $|q|_p<1$ will be
considered elsewhere). The main result of the present paper is the
following one.

\begin{thm}\label{mainres1} Let $p\geq3$, $|q|_p=1$ and $\sqrt[k]{1-q}\in\mathbb Q_p$. Assume that $X$ be a set defined as \eqref{X}.
If $|k|_p>p^{\frac{s(k-1)}{k}}|\theta-1|_p$ then $f_\theta$ over the
Julia set is conjugate to the full shift $(\Sigma,\sigma)$.
\end{thm}

We point out that the necessary and sufficient conditions for the
existence of $\sqrt[k]{1-q}$ in $\bq_p$ are given in \cite{MS13}. If
we take $q=2$ in \eqref{11} then the mapping is reduced to the Ising
mapping, so from Theorem \ref{mainres1} and \cite[Theorem 6.1]{MS13}
we immediately get the following corollary for the Ising mapping
which covers a result of \cite{MK17}.

\begin{cor} Let $p\geq3$ and $\frac{p-1}{(m,p-1)}$ is even (where $k=mp^s$, $s\geq 0$ and $(m,p-1)$ stands for common divisor of $m$ and $p-1$).
Assume that $X$ be a set defined as \eqref{X}.
If $|k|_p>|\theta-1|_p$ then the Ising mapping $f_\theta$ over the
Julia set is conjugate to the full shift $(\Sigma,\sigma)$.
\end{cor}

In the real numbers case, analogous results with rigorous proofs are
not known in the literature. In this direction, only numerical
analysis predict the existence of chaos (see for example,
\cite{ADH,BG,Monr}). The advantage of the non-Archimedeanity of the
norm allowed us rigorously to prove the existence of the chaos (in
Devaney's sense).

We note that the dynamical properties of the fixed points of some
$p$-adic rational maps have been studied in
\cite{ARS,B0,B1,FL4,KM1,RL}. However, the global dynamical structure
of rational maps on $\bq_p$ remains unclear.  Note that some
$p$-adic chaotic dynamical systems have been studied in
\cite{FL2,Wo}.

Note that, in the $p$-adic setting, due to lack of the convex
structure of the set of $p$- adic Gibbs measures, it was quite
difficult to constitute a phase transition with some features of the
set of $p$-adic Gibbs measures. The main main Theorem \ref{mainres1}
yields that the set of $p$-adic Gibbs measures is huge. Moreover,
the advantage of the present work allows to find lots of periodic
Gibbs measures.

\section{Preliminaries}

\subsection{$p$-adic numbers}

Let $\Q$ be the field of rational numbers. For a fixed prime number
$p$, every rational number $x\ne 0$ can be represented in the form
$x = p^r{n\over m}$, where $r, n\in \Z$, $m$ is a positive integer,
and $n$ and $m$ are relatively prime with $p$: $(p, n) = 1$, $(p, m)
= 1$. The $p$-adic norm of $x$ is given by
$$|x|_p=\left\{\begin{array}{ll}
p^{-r}\ \ \mbox{for} \ \ x\ne 0\\
0\ \ \mbox{for} \ \ x = 0.
\end{array}\right.
$$
This norm is non-Archimedean  and satisfies the so called strong
triangle inequality
$$|x+y|_p\leq \max\{|x|_p,|y|_p\}.$$
We recall a nice property of the norm, i.e. if $|x|_p >|y|_p$ then $|x\pm y|_p=|x|_p$.
Note that this is a crucial property which is proper to the non-Archimedenity of the norm.
The completion of $\Q$ with respect to the $p$-adic norm defines
the $p$-adic field
 $\Q_p$. Any $p$-adic number $x\ne 0$ can be uniquely represented
in the canonical form
\begin{equation}\label{ek}
x = p^{\gamma(x)}(x_0+x_1p+x_2p^2+\dots),
\end{equation}
where $\gamma(x)\in \Z$ and the integers $x_j$ satisfy: $x_0 > 0$,
$0\leq x_j \leq p - 1$ (see \cite{Ko,VVZ}). In this case $|x|_p =
p^{-\gamma(x)}$.

For each $a\in \bq_p$, $r>0$ we denote $$ B_r(a)=\{x\in \bq_p :
|x-a|_p< r\}.$$ Elements of the set $\mathbb{Z}_p=\{x\in \Q_p:
|x|_p\leq 1\}$ are called \textit{$p$-adic integers}.

The \textit{$p$-adic exponential} is defined by
$$\exp_p(x) =\sum^\infty_{n=0}{x^n\over n!},$$
which converges for every $x\in B_{p^{-1/(p-1)}}(0)$. Denote
$$
\mathcal E_p:=\left\{
\begin{array}{ll}
B_\frac{1}{4}(1), & \mbox{if }p=2;\\
B_\frac{1}{p}(1), & \mbox{if }p\neq2.
\end{array}
\right.
$$
This set is the range of the $p$-adic exponential function. It is
known \cite{Ko} the following fact.
\begin{lemma}\label{epproperty}
The set $\mathcal E_p$ has the following properties:\\
$(a)$ $\mathcal E_p$ is a group under multiplication;\\
$(b)$ $|a-b|_p<1$ for all $a,b\in\mathcal E_p$;\\
$(c)$ If $a,b\in\mathcal E_p$ then it holds
\[
|a+b|_p=\left\{\begin{array}{ll}
\frac{1}{2}, & \mbox{if }\ p=2\\
1, & \mbox{if }\ p\neq2.
\end{array}\right.
\]\\
$(d)$ If $a\in\mathcal E_p$, then
there is an element $h\in B_{p^{-1/(p-1)}}(0)$ such that
$a=\exp_p(h)$.
\end{lemma}

\begin{lemma}\label{alpbetgam}
Let $k\geq2$ and $p\geq3$. Then for any $\alpha,\beta\in\mathcal E_p$ there exists
$\gamma\in\mathcal E_p$ such that
\begin{equation}\label{abg}
\sum_{j=0}^{k-1}\alpha^{k-1-j}\beta^j=k\gamma
\end{equation}
\end{lemma}
\begin{proof}
Let $p\geq3$. Take any $\alpha,\beta\in\mathcal E_p$. If $\alpha=\beta$ then
we get
$$
\sum_{j=0}^{k-1}\alpha^{k-1-j}\beta^j=k\alpha^{k-1}
$$
According to Lemma \ref{epproperty} we get $\alpha^{k-1}\in\mathcal E_p$. So, in this case
we may take
$\gamma=\alpha^{k-1}$.

Now, we assume that $\alpha\neq\beta$.
Then according to Lemma \ref{epproperty} it holds $\frac{\alpha}{\beta}\in\mathcal E_p$.
For convenience we denote $\delta=\frac{\alpha}{\beta}$.

We show that $p^{1-n}<|n!|_p$ for any $n\geq2$. Indeed, it is clear
that
\begin{equation}\label{np}
|n!|_p>p^{-\frac{n}{p}(1+\frac{1}{p}+\frac{1}{p^2}\cdots)}=p^{-\frac{n}{p-1}}
\end{equation}
Noting $p\geq3$, we get $p^{-\frac{n}{p-1}}\geq p^{1-n}$ for any
$n\geq2$, which due to \eqref{np} implies $|n!|_p>p^{1-n}$ for any
$n\geq2$. Consequently, from $|\delta-1|_p\leq\frac{1}{p}$ one gets
\begin{equation}\label{delta}
|(\delta-1)^{n-1}|_p<|n!|_p,\ \ \ \mbox{for any }n\geq2
\end{equation}
On the other hand, we have
$$
\frac{|k!|_p}{|(k-n)!|_p}\leq|k|_p,\ \ \ \mbox{for any }2\leq n\leq k
$$
The last one with \eqref{delta} yields
\begin{equation}\label{ckn}
|C_k^n(\delta-1)^{n-1}|_p<|k|_p,\ \ \ \mbox{for any }2\leq n\leq k
\end{equation}
The equality
$$
\sum_{j=0}^{k-1}\alpha^{k-1-j}\beta^j=\alpha^{k-1}\sum_{j=0}^{k-1}\delta^j=\alpha^{k-1}\frac{\delta^k-1}{\delta-1}=
\alpha^{k-1}\left[k+\sum_{j=2}^kC_k^j(\delta-1)^{j-1}\right]
$$
with \eqref{ckn} and strong triangle inequality implies
$$
\left|\sum_{j=2}^kC_k^j(\delta-1)^{j-1}\right|_p<|k|_p
$$
Hence,
$$
1+\frac{1}{k}\sum_{j=2}^kC_k^j(\delta-1)^{j-1}\in\mathcal E_p
$$
Choosing $\gamma\in\mathcal E_p$ as follows
$$
\gamma=\alpha^{k-1}\left[1+\frac{1}{k}\sum_{j=2}^kC_k^j(\delta-1)^{j-1}\right]
$$
we obtain \eqref{abg}. This completes the proof.
\end{proof}

\subsection{$p$-adic subshift}
Let $f:X\to\mathbb Q_p$ be a map from a compact
open set $X$ of $\mathbb Q_p$ into $\mathbb Q_p$. We assume that (i) $f^{-1}(X)\subset X$;
(ii) $X=\cup_{j\in I}B_{r}(a_j)$ can be written as a finite disjoint
union of balls of centers $a_j$ and of the same radius $r$ such that for each $j\in I$ there is an integer
$\tau_j\in\mathbb Z$ such that
\begin{equation}\label{tau}
|f(x)-f(y)|_p=p^{\tau_j}|x-y|_p,\ \ \ \ x,y\in B_r(a_j).
\end{equation}
For such a map $f$, define its Julia set by
\begin{equation}\label{J}
J_f=\bigcap_{n=0}^\infty f^{-n}(X).
\end{equation}
It is clear that $f^{-1}(J_f)=J_f$ and then $f(J_f)\subset J_f$. We
will study the dynamical system $(J_f,f)$. The triple $(X,J_f,f)$ is
called a $p$-adic {\it repeller} if all $\tau_j$ in \eqref{tau} are
positive.
 For any $i\in I$, we let
$$
I_i:=\left\{j\in I: B_r(a_j)\cap
f(B_r(a_i))\neq\varnothing\right\}=\{j\in I: B_r(a_j)\subset
f(B_r(a_i))\}
$$
(the second equality holds because of the expansiveness and of the
ultrametric property). Then define a matrix $A=(a_{ij})_{I\times
I}$, called \textit{incidence matrix} as follows
$$
a_{ij}=\left\{\begin{array}{ll}
1,\ \ \mbox{if }\ j\in I_i;\\
0,\ \ \mbox{if }\ j\not\in I_i.
\end{array}
\right.
$$
If $A$ is irreducible, we say that $(X,J_f,f)$ is
\textit{transitive}. Here the irreducibility of $A$  means, for any
pair $(i,j)\in I\times I$ there is positive integer $m$ such that
$a_{ij}^{(m)}>0$, where $a_{ij}^{(m)}$ is the entry of the matrix
$A^m$.

Given $I$ and the irreducible incidence matrix $A$ as above. Denote
$$
\Sigma_A=\{(x_k)_{k\geq 0}: \ x_k\in I,\  A_{x_k,x_{k+1}}=1, \ k\geq
0\}
$$
which is the corresponding subshift space, and let $\sigma$ be the
shift transformation on $\Sigma_A$. We equip $\Sigma_A$ with a
metric $d_f$ depending on the dynamics which is defined as follows.
First for $i,j\in I,\ i\neq j$ let $\k(i,j)$ be the integer such
that $|a_i-a_j|_p=p^{-\k(i,j)}$. It clear that $\k(i,j)<\tau$. By
the ultra-metric inequality, we have
$$
|x-y|_p=|a_i-a_j|_p\ \ \ i\neq j,\ \forall x\in B_r(a_i), \forall
y\in B_r(a_j)
$$
For $x=(x_0,x_1,\dots,x_n,\dots)\in\Sigma_A$ and
$y=(y_0,y_1,\dots,y_n,\dots)\in\Sigma$, define
$$
d_f(x,y)=\left\{\begin{array}{ll}
p^{-\tau_{x_0}-\tau_{x_1}-\cdots-\tau_{x_{n-1}}-\k(x_{n},y_{n})}&, \mbox{ if }n\neq0\\
p^{-\k(x_0,y_0)}&, \mbox{ if }n=0
\end{array}\right.
$$
where $n=n(x,y)=\min\{i\geq0: x_i\neq y_i\}$. It is clear that $d_f$
defines the same topology as the classical metric which is defined
by $d(x,y)=p^{-n(x,y)}$.

\begin{thm}\cite{FL2}\label{xit} Let $(X,J_f,f)$ be a transitive $p$-adic weak repeller with incidence matrix $A$.
Then the dynamics $(J_f,f,|\cdot|_p)$ is isometrically conjugate to
the shift dynamics $(\Sigma_A,\sigma,d_f)$.
\end{thm}

\section{The fixed points of the Potts-Bethe function}

In this section, for a given a integer $q\geq2$ we study fixed
points of the Potts-Bethe mapping
\begin{equation}\label{func}
f_\theta(x)=\left(\frac{\theta x+q-1}{x+\theta+q-2}\right)^k,\ \ \
\theta\in\mathcal E_p.
\end{equation}
In what follows we assume that $p\geq3$ and $|q|_p=1$.

Let $x^{(0)}$ be a fixed point of an analytic function $f(x)$ and
$$
\lambda=\frac{\partial}{\partial x}f(x^{(0)}).
$$
The fixed point $x^{(0)}$ is called {\it attractive} if
$0<|\lambda|_p<1$, {\it indifferent} if $|\lambda|_p$, and {\it
repelling} if $|\lambda|_p>1$ (see \cite{AKh}).

Let $x^{(0)}$ be an attractive fixed point of the function $f$.
We define basin of attraction of this point as follows
$$
A(x^{(0)})=\{x: \lim_{n\to\infty}f^n(x)=x^{(0)}\}
$$
where $f^n=\underbrace{f\circ f\circ\dots\circ f}_n$

\begin{lemma}\label{contrac1}
Let $|q|_p=1$. Then $f_\theta(\mathcal E_p)\subset\mathcal E_p$ and
$f_\theta$ is a contraction on $\mathcal E_p$.
\end{lemma}
\begin{proof}
Let $|q|_p=1$. Pick any $x\in\mathcal E_p$. According to Lemma
\ref{epproperty} we get
$$
f_\theta(x)=\left(\frac{1+\frac{\theta x-1}{q}}{1+\frac{\theta-1+x-1}{q}}\right)^k\in\mathcal E_p
$$
which means that $f_\theta(\mathcal E_p)\subset\mathcal E_p$.

Using non-Archimedean norm's property one has
\begin{equation}\label{contf}
\frac{|f_\theta(x)-f_\theta(y)|_p}{|x-y|_p}=|k(\theta-1)|_p,\ \ \ \mbox{for any pair }x,y\in\mathcal E_p
\end{equation}
Noting $|k|_p\leq1$ and $|\theta-1|_p<1$ from \eqref{contf}
we obtain
$$
|f_\theta(x)-f_\theta(y)|_p<|x-y|_p
$$
which is equivalent to the contractivity of $f_\theta$ (In the real
case it is not true. But, in the $p$-adic case, a function $f: X\to
X$ for some compact subset $X$ of $\mathbb Q_p$ be a contraction if
and only if $|f(x)-f(y)|_p<|x-y|_p$ for any $x,y\in X$).
\end{proof}
Let us denote
$$
\begin{array}{ll}
B_{q,\theta}^{(1)}=\{x\in\mathbb Q_p: |x+q-1|_p>|\theta-1|_p\}\\[2mm]
B_{q,\theta}^{(2)}=\{x\in\mathbb Q_p: |x+q-1|_p=|\theta-1|_p\}\\[2mm]
B_{q,\theta}^{(3)}=\{x\in\mathbb Q_p: |x+q-1|_p<|\theta-1|_p\}
\end{array}
$$
Note that $\mathcal E_p\subset B_{q,\theta}^{(1)}$.
\begin{lemma}\label{contrac}
Let $p\geq3$. Then
$f_\theta(B_{q,\theta}^{(1)})\subset\mathcal E_p$.
\end{lemma}
\begin{proof}
Take any $x^*\in B_{q,\theta}^{(1)}$. Since $\theta\in\mathcal E_p$
and $|q|_p\leq1$ we find
$$
\theta+\frac{(\theta-1)(q-1)}{x^*+q-1}\in\mathcal E_p,\ \ \ 1+\frac{\theta-1}{x^*+q-1}\in\mathcal E_p
$$
Then according to Lemma \ref{epproperty} one gets
$$
f_\theta(x^*)=\left(\frac{\theta x^*+q-1}{x^*+\theta+q-2}\right)^k
=\left(\frac{\theta+\frac{(\theta-1)(q-1)}{x^*+q-1}}{1+\frac{\theta-1}{x^*+q-1}}\right)^k\in\mathcal E_p
$$
The arbitrariness of $x^*$ in $B_{q,\theta}^{(1)}$ yields that
$f_\theta(B_{q,\theta}^{(1)})\subset\mathcal E_p$.
\end{proof}

Let $|q-1|_p=p^{-s},\ s\geq0$.
We define the set
$$
\mbox{Sol}_p(x^k+q-1)=\left\{p^{-\frac{s}{k}}\xi\in\mathbb F_p: \xi^k+q-1\equiv0(\operatorname{mod }p^{s+1})\right\}
$$
If $\mbox{Sol}_p(x^k+q-1)\neq\varnothing$, we then denote
$\kappa_p:=|\mbox{Sol}_p(x^k+q-1)|$.
\begin{lemma}\label{tushinarsiz}
Let $p\geq3$ and $|q-1|_p=p^{-s},\ s\geq0$. Then
$\mbox{Sol}_p(x^k+q-1)\neq\varnothing$ if and only if there exists a
$p$-adic integer $\eta$ such that $|\eta^k+q-1|_p<|q-1|_p$.
\end{lemma}
\begin{proof}
From the definition of the set $\mbox{Sol}_p(x^k+q-1)$ we can easily
see that $\mbox{Sol}_p(x^k+q-1)\neq\varnothing$ if and only if
$|\eta^k+q-1|_p\leq\frac{1}{p^{s+1}}$ for some $p$-adic number
$\eta$. Since $|q-1|_p=p^{-s}$ one gets $|\eta^k+q-1|_p<|q-1|_p$.
From the non-Archimedean norm's
 property we obtain $|\eta|_p=|q-1|_p$ which means that $\eta$ is a $p$-adic integer.
\end{proof}

\begin{pro}\label{prop11}
Let $p\geq3$ and $|\theta-1|_p<|q-1|_p$. Assume that $\mbox{Sol}_p(x^k+q-1)=\varnothing$. Then $f_\theta$
has a unique fixed point $x_0=1$. Moreover, $A(x_0)=\mbox{Dom}(f_\theta)$.
\end{pro}
\begin{proof}
It is clear that $x_0=1$ is a fixed point for $f_\theta$. One has
$$
\frac{\partial}{\partial x}f_\theta(x_0)=\frac{k(\theta-1)}{\theta-1+q}
$$
The last one with $|q|_p=1$, $|k|_p\leq1$ implies
$|f'_\theta(x_0)|_p\leq|\theta-1|_p<1$. It means that $x_0$ is
attracting. Let us show that $A(x_0)=\mbox{Dom}(f_\theta)$. Indeed,
due to $\mbox{Sol}_p(x^k+q-1)=\varnothing$ from Lemma
\ref{tushinarsiz} one finds
$$
|y^k+q-1|_p\geq|q-1|_p,\ \ \ \mbox{for any }y\in\mathbb Q_p.
$$
In particularly, for any $x\in\mbox{Dom}(f_\theta)$ one has
$|f_\theta(x)+q-1|_p\geq|q-1|_p$. Since $|q-1|_p>|\theta-1|_p$ we
obtain $f_\theta(x^{(0)})\in B_{q,\theta}^{(1)}$. According to Lemma
\ref{contrac} one gets $f^2_\theta(x^{(0)})\in\mathcal E_p$.
Finally, the contractivity of $f_\theta$ on $\mathcal E_p$ and
$x_0\in\mathcal E_p$ yield that that $f^n(x)\to x_0$ as
$n\to\infty$. The arbitrariness of $x$ means that
$A(x_0)=\mbox{Dom}(f_\theta)$.
\end{proof}

Let us consider the case $\mbox{Sol}_p(x^k+q-1)\neq\varnothing$. For
a given $q\geq2$ with $|q-1|_p=p^{-s},\ s\geq0$ we define the set
\begin{equation}\label{X}
X=\bigcup_{i=1}^{\kappa_p} B_r(x_i),
\end{equation}
which is a finite union of disjoint balls. Here,
$r=\left|p^{\frac{s+k}{k}}(\theta-1)\right|_p$ and $x_i$ is defined
by \eqref{x_i1} if $s=0$ and by \eqref{x_i2} if $s\neq0$.
\begin{itemize}
\item[\textbf{Case.}] $s=0$
\begin{equation}\label{x_i1}
x_i=\left\{\begin{array}{ll}
1-q+\eta_i(\theta-1),& \mbox{if }\ \xi_i+q-1\not\equiv0(\operatorname{mod }p)\\[2mm]
1-q,& \mbox{if }\ \xi_i+q-1\equiv0(\operatorname{mod }p)
\end{array}
\right.
\end{equation}
where $\eta_i$ be a solution of
$$
\eta_i(\xi_i-1)+\xi_i+q-1\equiv0(\operatorname{mod }p)
$$
for a given $\xi_i\in\mbox{Sol}_p(x^k+q-1),\ i=\overline{1,\kappa_p}$.\\
\item[\textbf{Case.}]  $s>0$
\begin{equation}\label{x_i2}
x_i=
1-q+p^{\frac{s}{k}}\xi_i(\theta-1)
\end{equation}
where $\xi_i\in\mbox{Sol}_p(x^k+q-1),\ i=\overline{1,\kappa_p}$.
\end{itemize}

\begin{pro}\label{propdd}
Let $p\geq3$ and $|\theta-1|_p<|q-1|_p$. Assume that
$\mbox{Sol}_p(x^k+q-1)\neq\varnothing$. Then one has
$A(x_0)\supset\mbox{Dom}(f_\theta)\setminus X$.
\end{pro}

\begin{proof}
Take any $x\in\mbox{Dom}(f_\theta)\setminus X$. Note that
 $B_{q,\theta}^{(1)}\cap X=\varnothing$. So, first we consider a case
$x\not\in B_{q,\theta}^{(1)}\cup X$. Then, there exists a $p$-adic
integer $\eta$ such that
$$
x=1-q+\eta(\theta-1).
$$
Now consider two cases with respect to $s$.

\textbf{Case.} $s=0$. Suppose that
$\eta(\xi-1)+\xi+q-1\not\equiv0(\operatorname{mod }p)$ for any
solution $\xi$ of $x^k+q-1\equiv0(\operatorname{mod }p)$. Then we
get
$$
\left(\frac{1-q+\eta}{1+\eta}\right)^k+q-1\not\equiv0(\operatorname{mod }p)
$$
The last one with $|\theta-1|_p<1$ implies that
$$
|f_\theta(x)+q-1|_p=\left|\left(\frac{1-q+\eta+\eta(\theta-1)}{1+\eta}\right)^k+q-1\right|_p\geq1
$$
Hence, $f_\theta(x)\in B_{q,\theta}^{(1)}$.

 \textbf{Case.} $s>0$.
Assume that $|\eta|_p\neq p^{-\frac{s}{k}}$. Then by the
non-Archimedean norm's property one gets
$$
\left|1-q+\eta\theta\right|_p=\max\{|1-q|_p, |\eta|_p\}
$$
Noting $|\eta+1|_p=1$ we have
$$
|f_\theta(x)+q-1|_p=\left|\left(\frac{1-q+\eta\theta}{1+\eta}\right)^k+q-1\right|_p\geq|q-1|_p
$$
Since $|\theta-1|_p<|q-1|_p$ one gets $f_\theta(x)\in
B_{q,\theta}^{(1)}$.

Let $|\eta|_p=p^{-\frac{s}{k}}$. Then there exists a $p$-adic
integer $|\xi|_p=1$ such that $\eta=p^{\frac{s}{k}}\xi$. One has
$$
\theta+\frac{p^{\frac{s}{k}}(1-q)}{\xi}\in\mathcal E_p,\ \ \ \ 1+p^{\frac{s}{k}}\xi\in\mathcal E_p
$$
Then according to Lemma \ref{epproperty} we obtain
$$
\left(\frac{\theta+\frac{1-q}{p^{\frac{s}{k}}\xi}}{1+p^{\frac{s}{k}}\xi}\right)^k\in\mathcal E_p
$$
Hence,
$$
p^s\xi^k\left(\frac{\theta+\frac{1-q}{p^{\frac{s}{k}}\xi}}{1+p^{\frac{s}{k}}\xi}\right)^k\equiv
p^s\xi^k(\operatorname{mod }p^{s+1})
$$
From $\xi\not\in\mbox{Sol}_p(x^k+q-1)$ one finds
$\left|p^s\xi^k+q-1\right|_p=|q-1|_p$. Consequently, we have
$|f_\theta(x)+q-1|_p=|q-1|_p$, which yields $f_\theta(x)\in
B_{q,\theta}^{(1)}$.

Thus, we have shown that $f_\theta(x)\in B_{q,\theta}^{(1)}$ for any
$x\not\in B_{q,\theta}^{(1)}\cup X$. On the other hand, according to
Lemmas \ref{contrac1} and \ref{contrac} one has $A(x_0)\supset
B_{q,\theta}^{(1)}$. So, we conclude that
$A(x_0)\supset\mbox{Dom}(f_\theta)\setminus X$.
\end{proof}

\begin{pro}\label{prop12}
Let $p\geq3$ and $|\theta-1|_p<|q-1|_p=p^{-s},\ s\geq0$. If
$\mbox{Sol}_p(x^k+q-1)\neq\varnothing$. Then one has
\begin{equation}\label{fx-fy}
|f_\theta(x)-f_\theta(y)|_p=\frac{|k|_p}{p^{\frac{s(k-1)}{k}}|\theta-1|_p}|x-y|_p,\
\ \ \mbox{for any }x,y\in B_r(x_i).
\end{equation}
\end{pro}
\begin{proof}
For any pair $(x,y)\in\mathbb Q_p^2$ we have
\begin{equation}\label{fx-fy=}
f_\theta(x)-f_\theta(y)=\frac{(\theta-1)(\theta+q-1)
\sum_{j=0}^{k-1}[R(x)Q(y)]^{k-1-j}[R(y)Q(x)]^{j}}{[Q(x)Q(y)]^{k}}(x-y)
\end{equation}
where
$$
R(x)=\theta x+q-1,\ \ \ \ Q(x)=x+\theta+q-2
$$
\textbf{Case.} $s=0$. Pick any $z\in B_r(x_i)$. Suppose that
$x_i\neq1-q$. Then there exists an $\alpha_z\in p\mathbb Z_p$ such
that
$$
z=1-q+(\theta-1)(\eta_i+\alpha_z)
$$
We have
$$
\begin{array}{ll}
R(z)=(\theta-1)\left(1-q+\eta_i+\alpha_z+(\theta-1)(\eta_i+\alpha_z)\right)\\
Q(z)=(\theta-1)\left(1+\eta_i+\alpha_z\right)
\end{array}
$$
It follows from $q\not\equiv0(\operatorname{mod }p)$ and the
definition of $\eta_i$ that $|\eta_i+1|_p=|\eta_i+1-q|_p=1$. So,
there exist $p$-adic integers $\beta_z,\gamma_z\in p\mathbb Z_p$
such that
$$
R(z)=(\theta-1)(1-q+\eta_i)(1+\beta_z),\ \ \ \
Q(z)=(\theta-1)(1+\eta_i)(1+\gamma_z)
$$
It follows that $|R(z)|_p=|Q(z)|_p=|\theta-1|_p$. Plugging them into
\eqref{fx-fy=} and according to Lemma \ref{abg}
 one finds
$$
|f_\theta(z_1)-f_\theta(z_2)|_p=\frac{|k|_p}{|\theta-1|_p}|z_1-z_2|_p,\ \ \ \mbox{for any }z_1,z_2\in B_r(x_i)
$$
Assume that $x_i=1-q$. Then for any $z\in B_r(x_i)$ there exists an
$\alpha_z\in p\mathbb Z_p$ such that $z=1-q+\alpha_z(\theta-1)$. In
this case, we get
$$
\begin{array}{ll}
R(z)=(\theta-1)(1-q)(1+\frac{\alpha_z}{1-q})\\
Q(z)=(\theta-1)(1+\alpha_z)
\end{array}
$$
According to Lemma \ref{abg} from \eqref{fx-fy=} one has
$$
f_\theta(z_1)-f_\theta(z_2)=\frac{k(1-q)^{k-1}\gamma}{\theta-1}(x-y),\ \ \ \gamma\in\mathcal E_p
$$
for any $z_1,z_2\in B_r(x_i)$.

Hence,
$$
|f_\theta(z_1)-f_\theta(z_2)|_p=\frac{|k|_p}{|\theta-1|_p}|z_1-z_2|_p
$$
\textbf{Case.} $s>0$. Then for $z\in B_r(x_i)$ we can find
$$
\begin{array}{ll}
R(z)=p^{\frac{s}{k}}\xi_i(\theta-1)(1+\beta_z),\\
Q(z)=(\theta-1)(1+\gamma_z)
\end{array}
$$
Putting these into \eqref{fx-fy=} and using the non-Archimedean
norm's property one gets
$$
|f_\theta(z_1)-f_\theta(z_2)|_p=\frac{|k|_p}{p^{\frac{s(k-1)}{k}}|\theta-1|_p}|z_1-z_2|_p,\
\ \ \mbox{for any }z'_1,z'_2\in B_r(x_i).
$$
This completes the proof.
\end{proof}

\section{Proof of Theorem \ref{mainres1}}

In this section, we assume that $|q-1|_p=p^{-s},\ s\geq0$ and
$\kappa_p\geq2$.

\begin{pro}\label{prop13}
Let $p\geq3$ and $\sqrt[k]{1-q}\in\mathbb Q_p$. Assume that $X$ be a
set defined as \eqref{X}. Then a triple $(X,J_{f_\theta},f_\theta)$
is a $p$-adic repeller iff $|k|_p>p^{\frac{s(k-1)}{k}}|\theta-1|_p$.
\end{pro}
\begin{proof}
According to Proposition \ref{prop12} it is enough to show that
$f_\theta^{-1}(X)\subset X$.

Let us assume that $\sqrt[k]{1-q}\in\mathbb Q_p$. A function
$f_\theta$ has the following inverse branches on $X$:
$$
g_{\theta,i}(x)=\frac{(\theta+q-2)\hat{\xi}_i\sqrt[k]{\overline{x}}+1-q}{\theta-\hat{\xi}_i\sqrt[k]{\overline{x}}},\ \ i=\overline{1,\kappa_p}
$$
where $\hat{\xi}_i=\sqrt[k]{1-q}$ and $\sqrt[k]{\overline{x}}=\sqrt[k]{\frac{x}{1-q}}\in\mathcal E_p$.

We show that $g_{\theta,i}(x)\in B_r(x_i)$ for any $x\in X$.\\
\textbf{Case.} $s=0$. Take any $x\in X$. Then we get
\begin{equation}\label{g_i}
g_{\theta,i}-x_i=\frac{(\theta-1)\left[\eta_i(\hat{\xi}_i-1)+\hat{\xi}_i+q-1+(\eta_i+1)\hat{\xi}_i
(\sqrt[k]{\overline{x}}-1)-\eta_i(\theta-1)\right]}{1-\hat{\xi}_i-\hat{\xi}_i(\sqrt[k]{\overline{x}}-1)+\theta-1}
\end{equation}
One can see that
$$
\begin{array}{ll}
|\eta_i(\hat{\xi}_i-1)+\hat{\xi}_i+q-1|_p<1\\[2mm]
|\sqrt[k]{\overline{x}}-1|_p<1,\ \ |\theta-1|_p<1\\[2mm]
|\hat{\xi}_i|_p=|\eta_i|_p=|1-\hat{\xi}_i|_p=1
\end{array}
$$
Plugging these into \eqref{g_i} and using the strong triangle
inequality we obtain
$$
\left|g_{\theta,i}-x_i\right|_p<|\theta-1|_p
$$
which implies that $g_{\theta,i}(x)\in B_r(x_i)$. Hence, $f_\theta^{-1}(X)\subset X$.\\
\textbf{Case.} $s>0$. Take any $x\in X$. Then, we have
\begin{equation}\label{g_i2}
g_{\theta,i}-x_i=\frac{(\theta-1)\left[\hat{\xi}_i\sqrt[k]{\overline{x}}+q-1-p^{\frac{s}{k}}\xi_i
(\theta-\hat{\xi}_i\sqrt[k]{\overline{x}})\right]}{\theta-\hat{\xi}_i\sqrt[k]{\overline{x}}}
\end{equation}
Noting $\hat{\xi}_i=p^{\frac{s}{k}}\xi_i\alpha,\ \alpha\in\mathcal E_p$, we have
$$
\begin{array}{ll}
\left|\theta-\hat{\xi}_i\sqrt[k]{\overline{x}}\right|_p=1,\\[2mm]
\left|\hat{\xi}_i\sqrt[k]{\overline{x}}-p^{\frac{s}{k}}\xi_i
\theta\right|_p<p^{-\frac{s}{k}},\\[2mm]
\left|p^{\frac{s}{k}}\xi_i\hat{\xi}_i\sqrt[k]{\overline{x}}\right|_p=p^{-\frac{2s}{k}}.
\end{array}
$$
Putting the last ones into \eqref{g_i2} one finds
$$
|g_{\theta,i}(x)-x_i|_p\leq p^{-\frac{s}{k}-1}|\theta-1|_p
$$
From this, we immediately obtain $g_{\theta,i}(x)\in B_r(x_i)$ for
any $x\in X$.
\end{proof}

From this Proposition we immediately have the following
\begin{cor}\label{corol2}
Let $p\geq3$ and $\sqrt[k]{1-q}\in\mathbb Q_p$. Assume that $X$ be a
set defined as \eqref{X}. If
$|k|_p>p^{\frac{s(k-1)}{k}}|\theta-1|_p$ then for any
$i,j\in\{1,2,\dots,\kappa_p\}$ one has $B_r(x_i)\subset
f_\theta(B_r(x_j))$.
\end{cor}

\begin{proof}[Proof of Theorem \ref{mainres1}]
We have shown that under conditions $\sqrt[k]{1-q}\in\mathbb Q_p$,
$\kappa_p\geq2$ and $|k|_p>p^{\frac{s(k-1)}{k}}|\theta-1|_p$ the
$(X,J_{f_\theta},f_\theta)$ triple is a $p$-adic repeller. Moreover,
under these conditions incidence matrix $A$ for the function
$f_\theta: X\to \mathbb Q_p$ has the following form (it follows from
Corollary \ref{corol2})
$$
A=\left(\begin{array}{llll}
1 & 1 & \dots & 1\\
1 & 1 & \dots & 1\\
\vdots & \vdots & \ddots & \vdots\\
1 & 1 & \dots & 1
\end{array}
\right)
$$
This means that a triple $(X,J_{f_\theta},f_\theta)$ is a
transitive, hence according to Theorem \ref{xit} we conclude that
the dynamics $(J_{f_\t},f_\t,|\cdot|_p)$ is isometrically conjugate
to the shift dynamics $(\Sigma_A,\sigma,d_f)$.
\end{proof}

\appendix

\section{$p$-adic measure}

Let $(X,\cb)$ be a measurable space, where $\cb$ is an algebra of
subsets $X$. A function $\m:\cb\to \bq_p$ is said to be a {\it
$p$-adic measure} if for any $A_1,\dots,A_n\subset\cb$ such that
$A_i\cap A_j=\emptyset$ ($i\neq j$) the equality holds
$$
\mu\bigg(\bigcup_{j=1}^{n} A_j\bigg)=\sum_{j=1}^{n}\mu(A_j).
$$

A $p$-adic measure is called a {\it probability measure} if
$\mu(X)=1$.  A $p$-adic probability measure $\m$ is called {\it
bounded} if $\sup\{|\m(A)|_p : A\in \cb\}<\infty $. For more detail
information about $p$-adic measures we refer to
\cite{K3},\cite{KhN}.

\section{Cayley tree}

Let $\Gamma^k_+ = (V,L)$ be a semi-infinite Cayley tree of order
$k\geq 1$ with the root $x^0$ (whose each vertex has exactly $k+1$
edges, except for the root $x^0$, which has $k$ edges). Here $V$ is
the set of vertices and $L$ is the set of edges. The vertices $x$
and $y$ are called {\it nearest neighbors} and they are denoted by
$l=<x,y>$ if there exists an edge connecting them. A collection of
the pairs $<x,x_1>,\dots,<x_{d-1},y>$ is called a {\it path} from
the point $x$ to the point $y$. The distance $d(x,y), x,y\in V$, on
the Cayley tree, is the length of the shortest path from $x$ to $y$.
$$
W_{n}=\left\{ x\in V\mid d(x,x^{0})=n\right\}, \ \
V_n=\overset{n}{\underset{m=0}{\bigcup}}W_{m}, \ \ L_{n}=\left\{
l=<x,y>\in L\mid x,y\in V_{n}\right\}.
$$
The set of direct successors of $x$ is defined by
$$
S(x)=\left\{ y\in W_{n+1}:d(x,y)=1\right\}, x\in W_{n}.
$$
Observe that any vertex $x\neq x^{0}$ has $k$ direct successors and
$x^{0}$ has $k+1$.

\section{$p$-adic quasi Gibbs measure}

In this section we recall the definition of  $p$-adic quasi Gibbs
measure (see \cite{M12}).

Let $\Phi=\{1,2,\cdots,q\}$, here $q\geq 2$, ($\Phi$ is called a
{\it state space}) and is assigned to the vertices of the tree
$\G^k_+=(V,\Lambda)$. A configuration $\s$ on $V$ is then defined as
a function $x\in V\to\s(x)\in\Phi$; in a similar manner one defines
configurations $\s_n$ and $\w$ on $V_n$ and $W_n$, respectively. The
set of all configurations on $V$ (resp. $V_n$, $W_n$) coincides with
$\Omega=\Phi^{V}$ (resp. $\Omega_{V_n}=\Phi^{V_n},\ \
\Omega_{W_n}=\Phi^{W_n}$). One can see that
$\Om_{V_n}=\Om_{V_{n-1}}\times\Om_{W_n}$. Using this, for given
configurations $\s_{n-1}\in\Om_{V_{n-1}}$ and $\w\in\Om_{W_{n}}$ we
define their concatenations  by
$$
(\s_{n-1}\vee\w)(x)= \left\{
\begin{array}{ll}
\s_{n-1}(x), \ \ \textrm{if} \ \  x\in V_{n-1},\\
\w(x), \ \ \ \ \ \ \textrm{if} \ \ x\in W_n.\\
\end{array}
\right.
$$
It is clear that $\s_{n-1}\vee\w\in \Om_{V_n}$.

The (formal) Hamiltonian of $p$-adic Potts model is
\begin{equation}\label{ph}
H(\sigma)=J\sum_{\langle x,y\rangle\in L}
\delta_{\sigma(x)\sigma(y)},
\end{equation}
where $J\in B(0, p^{-1/(p-1)})$ is a coupling constant, and
$\delta_{ij}$ is the Kroneker's symbol.

A construct of a  generalized $p$-adic quasi Gibbs measure
corresponding to the model is given below.

Assume that  $\h: V\setminus\{x^{(0)}\}\to\bq_p^{\Phi}$ is a
mapping, i.e. $\h_x=(h_{1,x},h_{1,x},\dots,h_{q,x})$, where
$h_{i,x}\in\bq_p$ ($i\in\Phi$) and $x\in V\setminus\{x^{(0)}\}$.
Given $n\in\bn$, we consider a $p$-adic probability measure
$\m^{(n)}_{\h,\rho}$ on $\Om_{V_n}$ defined by
\begin{equation}\label{mu}
\mu^{(n)}_{\h}(\s)=\frac{1}{Z_{n}^{(\h)}}\exp\{H_n(\s)\}\prod_{x\in
W_n}h_{\s(x),x}
\end{equation}
Here, $\s\in\Om_{V_n}$, and $Z_{n}^{(\h)}$ is the corresponding
normalizing factor
\begin{equation}\label{ZN1}
Z_{n}^{(\h)}=\sum_{\s\in\Omega_{V_n}}\exp\{H_n(\s)\}\prod_{x\in
W_n}h_{\s(x),x}.
\end{equation}

In this paper, we are interested in a construction of an infinite
volume distribution with given finite-dimensional distributions.
More exactly, we would like to find a $p$-adic probability measure
$\m$ on $\Om$ which is compatible with given ones $\m_{\h}^{(n)}$,
i.e.
\begin{equation}\label{CM}
\m(\s\in\Om: \s|_{V_n}=\s_n)=\m^{(n)}_{\h}(\s_n), \ \ \ \textrm{for
all} \ \ \s_n\in\Om_{V_n}, \ n\in\bn.
\end{equation}

We say that the $p$-adic probability distributions \eqref{mu} are
\textit{compatible} if for all $n\geq 1$ and $\sigma\in
\Phi^{V_{n-1}}$:
\begin{equation}\label{comp}
\sum_{\w\in\Om_{W_n}}\m^{(n)}_{\h}(\s_{n-1}\vee\w)=\m^{(n-1)}_{\h}(\s_{n-1}).
\end{equation}
 This condition according to the Kolmogorov extension theorem (see \cite{KL}) implies the existence of a unique $p$-adic measure
$\m_{\h}$ defined on $\Om$ with a required condition \eqref{CM}.
Such a measure $\m_{\h}$ is said to be {\it a $p$-adic quasi Gibbs
measure} corresponding to the model \cite{M12,M13}. If one has
$h_x\in\ce_p$ for all $x\in V\setminus\{x^{(0)}\}$, then the
corresponding measure $\m_\h$ is called \textit{$p$-adic Gibbs
measure} (see \cite{MR1}).

By $Q\cg(H)$ we denote the set of all $p$-adic quasi Gibbs measures
associated with functions $\h=\{\h_x,\ x\in V\}$. If there are at
least two distinct generalized $p$-adic quasi Gibbs measures such
that at least one of them is unbounded, then we say that \textit{a
phase transition} occurs.

The following statement describes conditions on $h_x$ guaranteeing
compatibility of $\mu_{\bf h}^{(n)}(\sigma)$.

\begin{thm}\label{comp1}\cite{M12} The measures $\m^{(n)}_{\h}$, $
n=1,2,\dots$ (see \eqref{mu}) associated with $q$-state Potts model
\eqref{ph} satisfy the compatibility condition \eqref{comp} if and
only if for any $n\in \bn$ the following equation holds:
\begin{equation}\label{eq1}
\hat h_{x}=\prod_{y\in S(x)}{\mathbf{F}}(\hat \h_{y},\theta),
\end{equation}
here and below a vector $\hat \h=(\hat h_1,\dots,\hat
h_{q-1})\in\bq_p^{q-1}$ is defined by a vector
$\h=(h_1,h_1,\dots,h_{q})\in\bq_p^{q}$ as follows
\begin{equation}\label{H}
\hat h_i=\frac{h_i}{h_q}, \ \ \ i=1,2,\dots,q-1
\end{equation}
and mapping ${\mathbf{F}}:\bq_p^{q-1}\times\bq_p\to\bq_p^{q-1}$ is
defined by
${\mathbf{F}}(\xb;\theta)=(F_1(\xb;\theta),\dots,F_{q-1}(\xb;\theta))$
with
\begin{equation}\label{eq2}
F_i(\xb;\theta)=\frac{(\theta-1)x_i+\sum\limits_{j=1}^{q-1}x_j+1}
{\sum\limits_{j=1}^{q-1}x_j+\theta}, \ \ \xb=\{x_i\}\in\bq_p^{q-1},
\ \ i=1,2,\dots,q-1.
\end{equation}
\end{thm}

Let us first observe that the set
$(\underbrace{1,\dots,1,h}_m,1,\dots,1)$ ($m=1,\dots,q-1$) is
invariant for the equation \eqref{eq1}. Therefore, in what follows,
we restrict ourselves to one of such lines, let us say
$(h,1,\dots,1)$.

In \cite{MK16} to establish the phase transition, we considered
translation-invariant (i.e. $\h=\{\h_x\}_{x\in V\setminus\{x^0\}}$
such that $\h_x=\h_y$ for all $x,y$) solutions of \eqref{eq1}. Then
the equation \eqref{eq1} reduced to the following one
\begin{equation}\label{eq12}
h=f_\t(h),
\end{equation}
where
\begin{equation}\label{f}
f_\t(x)=\bigg(\frac{\t x+q-1}{x+\t+q-2}\bigg)^k.
\end{equation}

Hence, to establish the existence of the phase transition, when
$k=2$, we showed \cite{MR1} that \eqref{eq12} has three nontrivial
solutions if $q$ is divisible by $p$.  Note that full description of
all solutions of the last equation has been carried out in
\cite{RKh} if $k=2$ and in \cite{SA15} if $k=3$.

\end{document}